\documentclass[12pt]{amsart}
\usepackage{epsfig,color}
\usepackage{amssymb} 

\makeatother

\theoremstyle{definition}

\def\fnum{equation}
\newtheorem{Thm}[\fnum]{Theorem}
\newtheorem{Cor}[\fnum]{Corollary}

\newtheorem{Lem}[\fnum]{Lemma}

\numberwithin{equation}{section}

\def\RR{{\bold R}}

\def\CC{{\bold C }}
\newcommand{\dv}{{\text {div}}}
\newcommand{\e}{{\text {e}}}

\newcommand{\cL}{{\mathcal{L}}}

\newcommand{\eqr}[1]{(\ref{#1})}

\title{Parabolic frequency on manifolds}
\author[]{Tobias Holck Colding}%
\address{MIT, Dept. of Math.\\
77 Massachusetts Avenue, Cambridge, MA 02139-4307.}
\author[]{William P. Minicozzi II}%
 
\thanks{The  authors
were partially supported by NSF Grants DMS 1812142 and DMS 1707270.}

\email{colding@math.mit.edu  and minicozz@math.mit.edu}

\begin{document}

\maketitle

\begin{abstract}
We prove monotonicity of a parabolic frequency on manifolds.  This is a parabolic analog of Almgren's frequency function.  Remarkably we get monotonicity on all manifolds and no curvature assumption is needed.  When the manifold is Euclidean space and the drift operator is the Ornstein-Uhlenbeck operator this can been seen to imply Poon's frequency monotonicity for the ordinary heat equation.    Monotonicity of frequency is a parabolic analog of the 19th century Hadamard three circles theorem about log convexity of holomorphic functions on $\CC$.
From the monotonicity, we get parabolic unique continuation and backward uniqueness.  
\end{abstract}

\section{Introduction}

Bounds on growth for functions satisfying a PDE give  crucial information and have many consequences.   One of the oldest bounds of this type is Hadamard's three circles theorem for holomorphic functions.  
For elliptic equations, such as the Laplace equation, Almgren proved the monotonicity of a frequency function that measures the rate of growth, \cite{A}.   
Almgren's frequency played a fundamental role in his regularity results, \cite{A}, and other areas; see, e.g., \cite{GL}, \cite{Lo}.  Almgren's frequency was generalized to the heat equation by Poon, \cite{P}, who proved the
monotonicity of a parabolic frequency function.  The results of Almgren and Poon rely heavily on the scaling structure of $\RR^n$ (cf. \cite{CM1}) and do not extend globally to general manifolds.  
Here we prove a very general monotonicity for drift heat equations on any manifold and show that this general monotonicity implies the earlier one.  Part of the strength is the simplicity of the argument yet the power of the consequences.

Suppose that $(M,g)$ is a Riemannian manifold.  Let $\phi:M\to \RR$ be a smooth function and  define an operator $\cL_{\phi}$ (drift Laplacian) on vector-valued functions $u: M \to \RR^N$ by
\begin{align}  \label{e:defLphi}
\cL_{\phi}\,u=\Delta\,u-\langle\nabla u,\nabla \phi\rangle=\e^{\phi}\,\dv\,\left(\e^{-\phi}\,\nabla u\right)\, .
\end{align}
These operators play an important role in many parabolic problems; see, e.g., \cite{CM2}, \cite{CM3}.
The prime example of  $\cL_{\phi}$ is where $M=\RR^n$ with the flat metric, $\phi=\frac{|x|^2}{4}$ and $\cL_{\frac{|x|^2}{4}}\,u=\Delta\,u-\frac{1}{2}\,\langle \nabla u,x\rangle$ is the Ornstein-Uhlenbeck operator.  
We let $L^2_{\phi}$ and $W^{1,2}_{\phi}$ be the spaces of square integrable $\RR^N$-valued functions and Sobolev functions with respect to the weight 
$\e^{-\phi}$.   It follows from \eqr{e:defLphi} that $\cL_{\phi}$ is self-adjoint on $W^{1,2}_{\phi}$ with respect to the weighted volume
\begin{align}
\int \langle u\, , \cL_{\phi}\,v \rangle \,\e^{-\phi}=-\int \langle \nabla u,\nabla v\rangle\,\e^{-\phi}\, .
\end{align}

\vskip2mm
Suppose that $u:M\times [a,b]\to \RR^N$ is smooth and{\footnote{Some growth assumption is necessary to rule out the classical Tychonoff example.}}  $u , u_t \in W^{1,2}_{\phi}$ for each $t\in [a,b]$.  
Set
\begin{align}
I(t)&=\int |u|^2\,\e^{-\phi}\, ,\\
D(t)&=-\int |\nabla u|^2\,\e^{-\phi}=\int \langle u\, , \cL_{\phi}\,u \rangle \,\e^{-\phi}\, ,\label{e:D}\\
U(t)&=\frac{D}{I}\, .
\end{align}
Observe that with our convention $U$ is always non-positive.

The next theorem is a parabolic version of the classical Hadamard's three circle theorem{\footnote{The three circles theorem was stated and proven by J.E. Littlewood in 1912, but he   stated it as a known theorem.  Harald Bohr and Edmund Landau, in 1896, attribute the theorem to Jacques Hadamard; Hadamard did not publish a proof.}}
 for holomorphic functions:

\begin{Thm}   \label{t:hadamard}
When $(\partial_t-\cL_{\phi})\,u=0$, then $(\log I)'(t)=2\,U(t)$ and $\log I(t)$ is convex so $U'\geq 0$.    Moreover, when $U$ is constant, then $u(x,t)=\e^{U\,t}\,u(x,0)$ and $u(\cdot,0)$ is an eigenfunction of $\cL_{\phi}$ with eigenvalue $-U$.  
\end{Thm}

Poon, \cite{P}, proved a monotonicity that can be shown (see Section \ref{s:s1}) to follow from the special case of Theorem \ref{t:hadamard} when $M=\RR^n$, $N=1$ and $\phi=\frac{|x|^2}{4}$.    His monotonicity holds on manifolds with non-negative sectional curvature and parallel Ricci curvature which are exactly the assumptions needed to generalize Hamilton's work, \cite{H1}, \cite{H2}, from Euclidean space to manifolds.\footnote{See the discussion in \cite{P} after theorem 1.1' on page 522 and the remark on page 530.}  In contrast our monotonicity holds on any manifold and no curvature assumption is needed.

\vskip2mm
Theorem \ref{t:hadamard} has the following immediate consequences (recall that $U$ is non-positive):

\begin{Cor}  \label{c:UniqueCont}
If $u:M\times [a,b]\to \RR^N$ and $(\partial_t-\cL_{\phi})\,u=0$, then
\begin{align}  \label{e:firstpart}
I(b)\geq I(a)\,\e^{2\, U(a)\,(b-a)}\, .
\end{align}
In particular, if $u(\cdot,b)=0$, then $u \equiv 0$. 
\end{Cor}

\begin{proof}
By Theorem \ref{t:hadamard}
\begin{align}
\log I(b)-\log I(a)=\int_a^b(\log I)'(s)\,ds = 2\,  \int_a^b  U(s) \, ds \geq 2\, U(a)\,(b-a)\, .
\end{align}
\end{proof}

Equation \eqr{e:firstpart} can be thought of as a bound for the vanishing order at $\infty$ whereas the second part is a version of backward uniqueness.   The first part implies strong unique continuation at $\infty$.   That is, if $u$ vanishes to infinite order at $\infty$, then it vanishes.  We say that $u:M\times (a,\infty)\to \RR$ vanishes to infinite order at $\infty$ if $\lim_{t\to\infty}\e^{c\,t}\,I(t)=0$ for all constants $c$.   

\vskip2mm
Suppose more generally $u$ satisfies the equation:
\begin{align}
(\partial_t-\cL_{\phi}-\lambda)\,u=0\, .
\end{align}
Where $\phi$ is as above and $\lambda=\lambda (t)$ is a function depending on $t$ only.  By considering 
$v(x,t)=\e^{-\int_a^t \lambda (s)\,ds}\,u(x,t)$ and observing that $v$ satisfies \eqr{e:defLphi}.  It follows that our results apply to $v$ and hence we get a monotonicity for $u$.  

\vskip2mm
Our results holds also for more general operators (cf. \cite{ESS}, \cite{W}) where 
\begin{align}  \label{e:assumption}
|(\partial_t-\cL_{\phi})\,u|\leq C(t)\,(|u|+|\nabla u|)\, ,
\end{align}
and $C(t)$ is allowed to depend on $t$; see Theorem \ref{t:UniqueCont1} and Corollary \ref{c:secondUniqueCont}. 

\section*{Acknowledgement}
We are grateful to S. Brendle, R. Hamilton and S. Klainerman for discussions.

\section{Parabolic frequency on manifolds}   \label{s:s1}

\begin{proof}
(of Theorem \ref{t:hadamard}).  
Calculating and integrating by parts gives
\begin{align}
I'(t)&=2\,\int \langle u , \,u_t \rangle \,\e^{-\phi}=2\int \langle u\, ,\cL_{\phi}\,u \rangle \,\e^{-\phi}=-2\int |\nabla u|^2\,\e^{-\phi}=2\,D(t)\, .\label{e:I'}\\
D'(t)&=-2\,\int \langle \nabla u,\nabla u_t\rangle\,\e^{-\phi}=-2\,\int \langle \nabla u,\nabla \cL_{\phi}\,u\rangle\,\e^{-\phi}=2 \int |\cL_{\phi}\,u|^2\,\e^{-\phi}\, .\label{e:D'}
\end{align}
By \eqr{e:I'} and the definition of $U$ we get
\begin{align} \label{e:diffI}
(\log I)'(t)=2\,\frac{D(t)}{I(t)}=2\,U(t)\, .
\end{align}
Therefore, using \eqr{e:I'}, \eqr{e:D'} and \eqr{e:D}
\begin{align} \label{e:cs}
D'\,I-I'\,D&=\left(2\int |\cL_{\phi}\,u|^2\,\e^{-\phi}\right)\,\left(\int |u|^2\,\e^{-\phi}\right)-2\,D^2(t)\notag\\
&=\left(2\int |\cL_{\phi}\,u|^2\,\e^{-\phi}\right)\,\left(\int |u|^2\,\e^{-\phi}\right)-2\,\left(\int \langle  u, \,\cL_{\phi}\,u \rangle \,\e^{-\phi}\right)^2\geq 0\, .
\end{align}
Here the inequality follows from the Cauchy-Schwarz inequality.    Finally, from this we get 
\begin{align}  \label{e:monoU}
U'=\frac{D'\,I-I'\,D}{I^2}\geq 0\, .
\end{align}

When $U$ is constant $U'=0$ and we therefore have equality in the Cauchy-Schwarz inequality \eqr{e:cs}.  It follows that 
\begin{align}
\cL_{\phi}\,u=c(t)\,u\, .
\end{align}
Next to evaluate $c$ we observe that by the second equality in \eqr{e:D}
\begin{align}
D(t)=c(t)\int |u|^2\,\e^{-\phi}=c(t)\,I(t)\, .
\end{align}
It follows that $c(t)=U$ and $\cL_{\phi}\,u=U\,u$.    If we set 
\begin{align}
v(x,t)=\e^{-U\,t}\,u (x,t)\, ,
\end{align}
then we have that 
\begin{align}
\partial_tv=\e^{-U\,t}\,(-U\,u+\partial_tu)=\e^{-U\,t}\,(-U\,u+\cL_{\phi}\,u)=0\, .
\end{align}
From this the second claim follows.  
\end{proof}

There is a natural correspondence on $\RR^n$ between solutions of the ordinary heat equation and solutions of the drift heat equation:
Given  $u:\RR^n\times (-\infty,0)\to \RR$,  define $v(x,t)=u(\sqrt{-t}\,x,t)$, $w(x,s)=v(x,-\e^{-s})$ and $t=-\e^{-s}$.  We have the following:  

\begin{Lem}	\label{l:COV}
The function $w:\RR^n \times (-\infty,0)\to \RR$ defined as above satisfies
\begin{align} \label{e:drifteq}
(\partial_s-\cL_{\frac{|x|^2}{4}})\,w (x,s)= \e^{-s} \, \left( u_t - \Delta u \right)(\e^{ - \frac{s}{2}}x, - \e^{-s})  \, .
\end{align}
\end{Lem}

\begin{proof}
To prove \eqr{e:drifteq}, we  use the chain rule to get
\begin{align}	
\partial_tv&=-\frac{1}{2\,\sqrt{-t}}\,\langle \nabla u,x\rangle+u_t \, ,\\
\partial_sw&=-\frac{\sqrt{-t}}{2}\,\langle \nabla u,x\rangle-t\, u_t\, ,  \label{e:e121}  \\
\nabla w&= \sqrt{-t} \,\nabla u\, ,\\
\Delta\,w&=-t\,\Delta\,u\, .  \label{e:e123}
\end{align}
Combining  \eqr{e:e121}--\eqr{e:e123} gives \eqr{e:drifteq}.
\end{proof}

 Poon, \cite{P}, considered solutions $u:\RR^n\times (-\infty,0)\to \RR$ to the ordinary heat equation on Euclidean space.  He showed a monotonicity that is easily seen to be equivalent to that $s\to \log H(\e^{\frac{s}{2}})$ is convex, where  
\begin{align}
H(R)=(-4\,\pi\,R^2)^{-\frac{n}{2}}\int u^2 (y, -R^2)\,\e^{ - \frac{|y|^2}{4\,R^2}}\, .
\end{align}
The convexity of $ \log H(\e^{\frac{s}{2}})$ follows from Theorem \ref{t:hadamard} when $M=\RR^n$ and $\phi=\frac{|x|^2}{4}$.    To see this suppose $u_t = \Delta u$, so that  $(\partial_s-\cL_{\frac{|x|^2}{4}})\,w=0$ by Lemma  \ref{l:COV}.
Using the definition of $I_w(s)$ and making the change of variables $y= \e^{ - \frac{s}{2}} \, x$ and $R= \e^{-\frac{s}{2}}$
gives
\begin{align}
I_w(s)=\int u^2(\e^{-\frac{s}{2}}\,x,-\e^{-s})\,\e^{-\frac{|x|^2}{4}} \, dx = R^{-n} \,  \int u^2(y, - R^2) \, \e^{ - \frac{|y|^2}{4R^2}}\, dy = H (\e^{ \frac{s}{2}}) \, .
\end{align}
From this and Theorem \ref{t:hadamard} the convexity of $\log H(\e^{ \frac{s}{2}})$  follows.

\section{More general operators}

\begin{Thm}    \label{e:genLem}
If $u : M \times [a,b] \to \RR^N$ satisfies \eqr{e:assumption}, then
\begin{align}  
U'&\geq C^2 \,(U-1)\, ,\\
C^2&\geq \left[\log (1-U)\right]'\, .\label{e:consequence}
\end{align}
\end{Thm}

\begin{proof}
First we rewrite $D$ as follows
\begin{align}
	D =  \int \langle u  , \, \cL_{\phi}\, u \rangle \, \e^{-\phi}  =   \int \langle u \, , [ u_t - \frac{1}{2}\,(u_t - \cL_{\phi}\, u)] \rangle \, \e^{-\phi} -\frac{1}{2}\int \langle  u , \,(u_t-\cL_{\phi}\,u) \rangle \,\e^{-\phi} \, .
\end{align}
Differentiating $I(t)$ and rewriting gives
\begin{align}  \label{e:diffIgen}
I'(t)&=2\int \langle u, \,u_t \rangle \,\e^{-\phi}=2\int \langle u\, , \cL_{\phi}\,u \rangle \,\e^{-\phi}+2\int \langle  u , \,(u_t-\cL_{\phi}\,u) \rangle \,\e^{-\phi}\notag\\
&=2 \int \langle u ,  \, [ u_t - \frac{1}{2}\,(u_t - \cL_{\phi} u)] \rangle  \, \e^{-\phi}+\int \langle u , \,(u_t-\cL_{\phi}\,u) \rangle \,\e^{-\phi}\, .
\end{align}
Hence, 
\begin{align}  \label{e:secondterm}
I'(t)\,D(t)=2\,\left(\int \langle u , \, [ u_t - \frac{1}{2}\,(u_t - \cL_{\phi}\, u) ]  \rangle \, \e^{-\phi}\right)^2-\frac{1}{2}\,\left(\int \langle u , \,(u_t-\cL_{\phi}\,u) \rangle \,\e^{-\phi}\right)^2\, .
\end{align}
Differentiating $D(t)$ and integrating by parts gives
\begin{align} 
D'(t)=-2\,\int \langle \nabla u,\nabla u_t\rangle\,\e^{-\phi}  &= 2\int \langle u_t  , \, \cL_{\phi}\, u \rangle \, \e^{-\phi} = 2\int \langle u_t , \, ( u_t - [u_t - \cL_{\phi}\, u]) \rangle\, \e^{-\phi} \notag \\
	&= 
	2 \int \left\{ |u_t - \frac{1}{2} \, [ u_t - \cL_{\phi}\, u]|^2 - \frac{1}{4} \, | u_t - \cL_{\phi}\, u|^2 \right\} \, \e^{-\phi} \, .
\end{align}
So
\begin{align}
D'(t)\,I(t)= 2\, I(t) \, \int |u_t - \frac{1}{2} \, [ u_t - \cL_{\phi}\, u]|^2\, \e^{-\phi}  
- \frac{I(t)}{2} \,\int | u_t - \cL_{\phi}\, u|^2 \, \e^{-\phi}\, .   \label{e:firstterm}
\end{align}
Combining \eqr{e:secondterm}  and \eqr{e:firstterm} and using the Cauchy-Schwarz inequality, \eqr{e:assumption} and the elementary inequality $(a+b)^2\leq 2\,(a^2+b^2)$ gives
\begin{align}  \label{e:keyinequality}
D'\,I-I'\,D&= 
	2 \,\left[ \int |u|^2 \, \e^{-\phi} \, \int |u_t - \frac{1}{2} \, [ u_t - \cL_{\phi}\, u] |^2  \, \e^{-\phi}-\left(\int \langle u , \, [ u_t - \frac{1}{2}\,(u_t - \cL_{\phi}\, u)] \rangle \, \e^{-\phi}\right)^2\right]\notag\\
	&- \frac{I(t)}{2} \int | u_t - \cL_{\phi}\, u|^2 \, \e^{-\phi}+\frac{1}{2}\,\left(\int \langle u , \,(u_t-\cL_{\phi}\,u) \rangle \,\e^{-\phi}\right)^2\\
	&\geq - \frac{I(t)}{2}\int |u_t - \cL_{\phi}\, u|^2 \, \e^{-\phi}\geq -\frac{C^2\,I(t)}{2}\int (|u|+|\nabla u|)^2\,\e^{-\phi}\notag
	\geq -C^2 \,I(t)\,(I(t)-D(t))\, .\notag
\end{align}
Dividing both sides by $I^2(t)$ gives the first claim.  
The second follows from the first.
\end{proof}

This leads to the following generalization of Corollary \ref{c:UniqueCont}:

\begin{Cor}  \label{t:UniqueCont1}
If $u:M\times [a,b]\to \RR^N$ satisfies \eqr{e:assumption} then
\begin{align}
I(b)\geq I(a)\,\exp\,\left((b-a)\,  (2 + \sup_{[a,b]} C) \,  \left[\exp\, \left(\int_a^b C^2(s)\,ds\right)\,[U(a)-1]+1-\frac{3}{2}\,\sup_{[a,b]}\,C\right]\right)\, . \notag
\end{align}
In particular, if $u(\cdot,b)=0$, then $u\equiv 0$. 
\end{Cor}

\begin{proof}
It follows from \eqr{e:diffIgen}, the Cauchy-Schwarz inequality and the elementary inequality $a\leq \frac{1}{2}\,\left(a^2+1\right)$ applied to $a=\sqrt{-U}$ that 
\begin{align}  \label{e:genlogI'}
(\log I)'&\geq 2\,U-\frac{C}{I}\,\int |u|\,(|u|+|\nabla u|)\,\e^{-\phi}\geq 2\,U-C\,(1+\sqrt{-U})\notag\\
&\geq \left(2+\frac{C}{2}\right)\,U-\frac{3\,C}{2}\, .
\end{align}
From this we get that 
\begin{align}    \label{e:difference}
\log I(b)-\log I(a)&=\int_a^b (\log I)'(s)\,ds\notag\\
&\geq \frac{1}{2}\,\left(4+\sup_{[a,b]}\, C\right)\int_a^b U(s)\,ds-\frac{3}{2}\,\sup_{[a,b]}\,C\,(b-a)\, .
\end{align}
From \eqr{e:consequence} we get that for $s\in [a,b]$
\begin{align}
\log (1-U(s))\leq \log (1-U(a))+\int_a^s C^2(r)\,dr \leq  \log (1-U(a))+\int_a^b C^2(s)\,ds\, .
\end{align}
Therefore
\begin{align}
U(s)\geq \exp\, \left(\int_a^s C^2(s)\,ds\right)\,(U(a)-1) +1\, .
\end{align}
Inserting this lower bound in \eqr{e:difference} and integrating gives
\begin{align}
\log I(b)-\log I(a)
\geq (b-a)\,\left[\exp\, \left( (2 + \sup_{[a,b]} C) \, \int_a^b C^2(s)\,ds\right)\,[U(a)-1]+1-\frac{3}{2}\,\sup_{[a,b]}\,C\right]\, . \notag
\end{align}
\end{proof}

Recall that that $u:M\times (a,\infty)\to \RR^N$ vanishes to infinite order at $\infty$ if $\lim_{t\to\infty}\e^{c\,t}\,I(t)=0$ for all constants $c$.    
Theorem \ref{t:UniqueCont1} implies the following strong unique continuation at $\infty$:

\begin{Cor}  \label{c:secondUniqueCont}
Suppose that $\sup\, C+\int_a^{\infty}C^2(s)\,ds<\infty$ and $u:M\times [a,\infty)\to \RR^N$ is a solution of \eqr{e:assumption} that vanishes to infinite order at $\infty$, then $u$ vanishes.
\end{Cor}

 This corollary implies the unique continuation of Poon, \cite{P}, who considered functions $u$ on $\RR^n$ into $\RR$ with
 \begin{align}
 		 u_t - \Delta u  = \langle b (x,t) , \nabla u \rangle + c(x,t) \, u \, ,
 \end{align}
 where $|b| + |c| \leq C$ is uniformly bounded (cf. \cite{L}).
 We will see that the results here apply more generally to functions $u$ satisfying the differential inequality
 \begin{align}
 	\left| u_t - \Delta u \right| \leq C \, (|u| + |\nabla u|) \, .
 \end{align}
 Applying the transformation in Lemma \ref{l:COV} to $u$, we get a function $w(y,s)$ with 
  \begin{align}
 	\left| \left(\partial_s - \cL_{\frac{|y|^2}{4}}\right) \, w \right| &= \e^{-s} \, \left| ( \partial_t - \Delta) \, u \right| \leq C \, \e^{-s} \left( |u| + |\nabla u| \right) \notag \\
	&\leq C \, \e^{-s} |w| + C \, \e^{ - \frac{s}{2} } \, |\nabla w|
	\, .
\end{align}
Since $\int_0^{\infty}\e^{- {s}}\,ds<\infty$, Corollary \ref{c:secondUniqueCont} applies.    Exponential decay of order $c$, i.e., decay like $\e^{-c\,s}$, corresponds to polynomial decay $t^c$ in the transformed variable $t=-\e^{-s}$.  

\subsection{Without $u$ term}

In this subsection we assume that $u:M\times [a,b]\to \RR^N$ satisfies
\begin{align}    \label{e:stronger}
|(\partial_t-\cL_{\phi})\,u|\leq C(t)|\nabla u|\, .
\end{align}
In this case we get better estimates when $U(a)$ is small.  
It follows from \eqr{e:keyinequality}, with obvious simplifications in the second to last inequality from using \eqr{e:stronger} in place of \eqr{e:assumption}, that
$U'\geq \frac{C^2}{2}\,U$ or, 
  equivalently,
\begin{align}
[\log (-U)]'\leq \frac{C^2}{2}\, .
\end{align}
We therefore get that
\begin{align}
U(s)\geq U(a)\,\exp\,\left(\frac{1}{2}\int_a^s C^2(\tau)\,d\tau\right)\, .
\end{align}
With similar simplifications in \eqr{e:genlogI'} we get that for $s\in [a,b]$
\begin{align}
(\log I)'&\geq  2\,U-C\,\sqrt{-U}\notag\\
&\geq 2\,U(a)\,\exp\,\left(\frac{1}{2}\int_a^b C^2(\tau)\,d\tau\right)
-C\,\sqrt{-U(a)}\,\exp\left(\,\frac{1}{4}\int_a^b C^2(\tau)\,d\tau\right)\, .
\end{align}
Integrating gives
\begin{align}
I(b)\geq I(a)\,\exp\,\left[(b-a)\,\left\{2\,U(a)\,\exp\,\left(\frac{1}{2}\int_a^b C^2(\tau)\,d\tau\right)
-C\,\sqrt{-U(a)}\,\exp\left(\,\frac{1}{4}\int_a^b C^2(\tau)\,d\tau\right)\right\}\right]\, .\notag
\end{align}

\end{document}